\documentclass[11pt]{article}
\usepackage{amsmath}
\usepackage{amsthm}
\usepackage{amssymb}
\usepackage{graphicx}
\usepackage{multirow} 
\usepackage{tabularx}
\newcommand{\Z}{\mathbb{Z}}
\newcolumntype{R}{>{\raggedleft\arraybackslash}X}
\newcolumntype{L}{>{\raggedright\arraybackslash}X}
\usepackage[a4paper,margin=3cm,right=3cm,includehead]{geometry}
\newtheorem{theorem}{Theorem}

\newtheorem{definition}[theorem]{Definition}

\title{ \bf Improved upper bounds for the order of some classes of Abelian Cayley and circulant graphs of diameter two}
\author{Robert. R. Lewis\\[-3pt]
\small Department of Mathematics and Statistics\\[-3pt]
\small The Open University\\[-3pt]
\small Milton Keynes, UK\\[-3pt]
\small \texttt{robert.lewis@open.ac.uk}}

\date{\small {18th February 2015} }



\makeatletter
\renewcommand\section{\@startsection {section}{1}{\z@}%
                                   {-2.5ex \@plus -1ex \@minus -.2ex}%
                                   {1.3ex \@plus.2ex}%
                                   {\normalfont\bf}}
\makeatother

\begin{document}
\maketitle

\begin{abstract}

In the degree-diameter problem for Abelian Cayley and circulant graphs of diameter 2 and arbitrary degree $d$ there is a wide gap between the best lower and upper bounds valid for all $d$, being quadratic functions with leading coefficient 1/4 and 1/2 respectively. Recent papers have presented constructions which increase the coefficient of the lower bound to be at or just below 3/8, but only for sparse sets of degree $d$ related to primes of specific congruence classes. By applying results from number theory these constructions can be extended to be valid for every degree above some threshold, establishing an improved asymptotic lower bound approaching 3/8. The constructions use the direct product of the multiplicative and additive subgroups of a Galois field and a specific coprime cyclic group. By generalising this method an improved upper bound, with quadratic coefficient 3/8, is established for this class of construction of Abelian Cayley and circulant graphs. Analysis of the order of the known extremal diameter 2 circulant graphs, up to degree 23, is shown to provide tentative support for a quadratic coefficient of 3/8 for the asymptotic upper bound for the order of general diameter 2 circulant graphs of arbitrary degree.


\end{abstract}


\section{Introduction}
The degree-diameter problem is to identify extremal graphs, having the largest possible number of vertices $n(d,k)$ for a given maximum degree $d$ and diameter $k$. For most categories of graphs, other than for small degree and diameter, this remains an open question. This paper considers the degree-diameter problem for undirected Cayley graphs of diameter 2 and arbitrary degree, of cyclic groups and of Abelian groups in general. It concerns the wide gap between the current best lower and upper bounds, being quadratic functions with leading coefficient 1/4 and 1/2 respectively. Cayley graphs of Abelian groups are more simply called \it Abelian Cayley graphs\rm, and we denote Cayley graphs of cyclic groups by \it circulant graphs\rm.

The best upper bound for circulant graphs of degree $d$ and diameter $k$ is also the best upper bound for the order of Abelian Cayley graphs of the same degree and diameter, $M_{AC}(d,k)$, given by Dougherty and Faber \cite{Dougherty}:

\[M_{AC}(d,k) =
\begin{cases}
\sum _{i=0}^f 2^i \binom {f}{i} \binom {k}{i} &\mbox{ for even } d, \mbox{ where } f=d/2\\
\sum _{i=0}^f 2^i \binom {f}{i}\{ \binom {k}{i}+\binom {k-1}{i} \} &\mbox { for odd } d, \mbox{ where } f=(d-1)/2.
\end{cases}
\]

Hence for diameter 2 and any degree $d$ we have $M_{AC}(d,2)=\lfloor \frac{1}{2} d^2+d+1\rfloor$.
Dougherty and Faber also gave a lower bound $L_{AC}(d,2)$ for circulant graphs of diameter 2 and arbitrary degree:
$L_{AC}(d,2)=\frac{1}{4}d^2+2d+\delta$, where $\delta$ is a constant depending on $d\pmod{4}$.
The gap between the quadratic coefficients of the lower bound, $1/4$, and the upper bound, $1/2$, is disappointingly large.

Recently Macbeth, \v{S}iagiov\'a and \v{S}ir\'{a}\v{n} established a better lower bound for an infinite but sparse set of degrees \cite{Macbeth}. Their family of solutions has order $n=9(d^2+d-6)/25$ for degree $d=5p-3$, where $p$ is a prime with $p\equiv2 \pmod3$, giving a quadratic coefficient of $9/25$ or 0.360. This was achieved by constructing the direct product of three coprime cyclic groups $F^*\times F^+\times \Z_9$, where $F=GF(p)$ is the Galois field of order $p$, with additive group $F^+$ and multiplicative group $F^*$, and selecting an appropriate generating set for the Cayley graph. Vetr\'ik extended this method to establish a slightly improved lower bound for a different infinite and sparse set of degrees \cite {Vetrik}. The graphs have order $n=13(d^2-2d-8)/36$ for degree $d=6p-2$, where $p$ is a prime, $p\ne 13$, and $p \not \equiv 1 \pmod{13}$, giving a quadratic coefficient of $13/36$ or about 0.361. The construction also involves the direct product of three coprime cyclic groups, $F^*\times F^+\times \Z_{13}$.

By relaxing the specification of the type of graph from circulant to any Abelian Cayley, an improved lower bound with a quadratic coefficent of $3/8$ or 0.375 was identified by Macbeth, \v{S}iagiov\'a and \v{S}ir\'{a}\v{n}, \cite{Macbeth}. This again involves the direct product of three cyclic groups, but this time they are not coprime so that the group is Abelian but not cyclic. The construction uses $F^*\times F^+\times \Z_6$, where $F=GF(q)$ is the Galois field of order an odd prime power $q$.

The question arises whether other constructions of a similar form might provide improved lower bounds, with quadratic coefficient above $13/36$ for an infinite set of circulant graphs or above $3/8$ for Abelian Cayley graphs. In the following sections we consider generalisations of this method for diameter 2 circulant and Abelian Cayley graphs, followed by a similar construction based on the direct product of three unrestricted cyclic groups.

For circulant graphs of diameter 2, employing the method of construction with the direct product of the additive and multiplicative groups of a Galois field and cyclic group of any order, it proves to be impossible to achieve a quadratic coefficient better than $3/8$, thus improving the upper bound for this construction method from $1/2$ to $3/8$. For Abelian Cayley graphs constructed in the same way the upper bound on the quadratic coefficient is also improved to a value of 3/8. Also the quadratic coefficient upper bound for the direct product of three unconstrained cyclic groups with the given construction method is improved to equal the general lower bound of $1/4$.

Using a recently discovered property of prime numbers in congruence classes, the constructions may be extended to be valid for every degree above some threshold, with slightly reduced quadratic coefficients. A recent paper by Cullinan and Hajir \cite{Cullinan} defines a method of identifying a bound on the size of interval which will always contain at least one prime of a prescribed congruence class. For any given degree, the size of the interval determines a maximum degree below which there is a valid prime solution. With this approach, the asymptotic value of the quadratic coefficient exceeds 0.358 for the two circulant graph construction and remains at 0.375 to three significant figures for the Abelian Cayley graph construction.


\section{Circulant graphs for groups of the form $G=F^* \times F^+ \times \Z_n$}

We first consider circulant graphs and a generalisation of the approach taken by Macbeth, \v{S}iagiov\'a and \v{S}ir\'{a}\v{n} and by Vetr\'ik based on the cyclic group $G=F^* \times F^+ \times \Z_n$, where $F=GF(p)$ for prime $p$, and $n$ is 9 and 13 respectively. For the generalisation we will consider this group for any odd $n \in \mathbb{N}$. Beforehand some important component sets of the connection set $X$ are defined.

\begin{definition}
For any $n \in \mathbb{N}, x\in F^*, y\in F^+, u, v, w\in \Z_n$,
let $a_u(x)=(x,x,u)$, $b_v(x)=(x,0,v)$ and $c_w(y)=(1,y,w)$, 
with $a^{-1}_u(x)=(x^{-1},-x,-u)$, $b^{-1}_v(x)=(x^{-1},0,-v)$, and $c^{-1}_w(y)=(1,-y,-w)$.
For any $u, v, w\in \Z_n$ we also define $A_u=\{ a_u(x)$, $a^{-1}_u(x)$: $x\in F^*\}$, $B_v=\{ b_v(x)$, $b^{-1}_v(x)$: $x\in F^*\}$, $C_w^*=\{ c_w(y)$, $c^{-1}_w(y)$: $y\in F^*\}$ and $C_w^+=\{ c_w(y)$, $c^{-1}_w(y)$: $y\in F^+\}$.
\end{definition}

Clearly for $u, v, w\ne 0$, $A_u$, $B_v$ and $C_w^*$ have size $2(p-1)$ and $C_w^+$ has size $2p$. For $u, v, w=0$ we see that $A_0$ has size $2(p-1)$, $B_0$ and $C_0^*$ have size $p-1$ and $C_0^+$ has size $p$.

In Macbeth, \v{S}iagiov\'a and \v{S}ir\'{a}\v{n}'s construction with $n=9$ the connection set $X$ is comprised of the sets $A_1, B_3$ and $C_0^*$ along with two other elements, and hence $|X|=5p-3$. It is relatively straightforward to prove that any element of $G$ can be expressed as the sum of at most two elements of $X$ so that the resultant Cayley graph $C=(G,X)$ has diameter 2. As the degree of the graph $d=|X|$, we have $p=(d+3)/5$. Thus the order of the Cayley graph $|G|=9p(p-1)=(9/25)(d+3)(d-2)$, giving the quadratic coefficient $9/25$. A necessary condition for the construction is that every element of $\Z_9$ can be realised as the sum or difference of two of the subscripts of $A_1, B_3$ and $C_0^*$, that is $\pm 1, \pm 3$ and 0. It is also necessary that each of these pairs of sets generate all but at most a constant number of the elements of $F^* \times F^+$, with the exceptions covered separately. In Vetr\'ik's construction with $n=13$, the connection set $X$ includes the sets $A_1, B_3$ and $C_4^+$ along with two other elements, thus $|X|=6p-2$. Again we find that any element of $G$ can be expressed as the sum of at most two elements of $X$ so that the resultant Cayley graph $C=(G,X)$ has diameter 2. In this case we have $p=(d+2)/6$. Thus the order of the Cayley graph $|G|=13p(p-1)=(13/36)(d+2)(d-4)$.

In both cases there is a relation between the degree $d$ and the prime number $p$ of the form $d=lp+\delta$ for constants $l, \delta$ (with $l=5$ and $l=6$ respectively) generating a graph of order $|G|=(n/l^2)d^2+O(d)$. For the generalisation of this approach we take the component sets $A_u, B_v, C_w^+$. For any $l\ge 3$ we consider different values of $n$ and for each $n$, the corresponding family of cyclic groups $G_l=F^* \times F^+ \times \Z_n$ for any prime $p$ such that $p, p-1, n$ are pairwise coprime. Now for each $G_l$ we consider generating sets $X$ comprised of all possible combinations of the sets $A_u, B_v, C_w^+$, along with a fixed number of other elements of $G_l$ such that $|X|=lp+\delta$ for some fixed $\delta$, with the condition that the Cayley graph $C(G_l,X)$ has diameter 2. There is no value in including $A_0$ in $X$ as any element $(x,y,0)$ is the sum of two elements of $A_u$ for any $u\ne 0$, as we shall see later. There is no material difference between including $B_0$ or $C_0^+$ and no value in including both. So we will consider $C_0^+$ as the only set with subscript 0 for possible inclusion in $X$. Thus if $l$ is odd $C_0^+$ in included in $X$ and if $l$ is even $C_0^+$ is not included. For any $l$ we define $n_l$ to be the largest value of $n$ for which such a Cayley graph $C(G,X)$ exists for all admissable $p$, and define $G_l=F^* \times F^+ \times \Z_{n_l}$. It follows that the order of $C(G_l,X)$ is $(n_l/l^2)d^2+O(d)$, and we denote this order by $CC_l(d,2)$. Finally we define the class of groups ${\cal G}=\{G_l: l\ge 3\}$ and $CC_{\cal G}(d,2)=\sup _{l\ge 3} {CC_l(d,2)}$. Theorem \ref{theorem:k2a} establishes an improved upper bound for such graphs, with quadratic coefficient 3/8.

\begin{theorem}
With $CC_{\cal G}(d,2)$ as defined above, $CC_{\cal G}(d,2) \le (3/8) d^2 + O(d)$.
\label{theorem:k2a}
\end{theorem}
\begin{proof}

Let $F^*$ be the multiplicative group and let $F^+$ be the additive group of the Galois field $GF(p)$, where $p$ is a prime such that $(p,n)=1$ and $(p-1,n)=1$, so that $p$, $p-1$ and $n$ are pairwise coprime. Let $G=F^* \times F^+ \times \Z_n$. Since $F^*$, $F^+$ and $\Z_n$ are cyclic groups of coprime order, the group $G$ is also cyclic. Let $0$ denote the identity in $F^+$ and $\Z_n$, and $1$ the identity in $F^*$.

Now consider the Cayley graph, $C(G,X)$, of the group $G$ with a connection set $X$ which includes the union of $A_u$, $B_v$ and $C_w$ for multiple non-zero values of $u$, $v$ and $w$, numbering $m$ in total. Consider $A_u$ for $u \in U$, $B_v$ for $v \in V$ and $C_w$ for $w \in W$ where $U, V$ and $W$ are index sets with $g_a=|U|$, $g_b=|V|$ and $g_c=|W|$, so that $g_a+g_b+g_c=m$.   $X$ is constructed to be inverse-closed so that $C(G,X)$ is an undirected circulant graph of degree $d=|X|$. Later we will also consider including the inverse-closed set $C_0$, and we define $l=2m$ if $C_0$ is not in $X$ and $l=2m+1$ if $C_0$ is in $X$. We now investigate how elements of $G$ may be constructed from pairs of the sets $A_u$, $B_v$, $C_w$.

First consider $B_v$ and $C_w$ for $v \in V$ and $w \in W$. For any $x\in F^*$, $y\in F^+$ we have
\[
\begin{array} {l l l}
(x,y,v+w) & =(x,0,v)(1,y,w) & =b_v(x) c_w(y) \\
(x,y,v-w) & =(x,0,v)(1,y,-w) & =b_v(x) c^{-1}_w(-y) \\
(x,y,-v+w) & =(x,0,-v)(1,y,w) & =b^{-1}_v(x^{-1}) c_w(y) \\
(x,y,-v-w) & =(x,0,-v)(1,y,-w) & =b^{-1}_v(x^{-1}) c^{-1}_w(-y) 
\end{array}
\]
Next consider $A_u$ and $C_w$ for $u \in U$ and $w \in W$. For any $x\in F^*$, $y\in F^+$ we have
\[
\begin{array} {l l l}
(x,y,u+w) & =(x,x,u)(1,y-x,w) & =a_u(x) c_w(y-x) \\
(x,y,u-w) & =(x,x,u)(1,y-x,-w) & =a_u(x) c^{-1}_w(x-y) \\
(x,y,-u+w) & =(x,-x^{-1},-u)(1,y+x^{-1},w) & =a^{-1}_u(x^{-1}) c_w(y+x^{-1}) \\
(x,y,-u-w) & =(x,-x^{-1},-u)(1,y+x^{-1},-w) & =a^{-1}_u(x^{-1}) c^{-1}_w(-y-x^{-1}) 
\end{array}
\]
Now consider $A_u$ and $B_v$ for $u \in U$ and $v \in V$. For any $x\in F^*$, $y\in F^*$ we have
\[
\begin{array} {l l l}
(x,y,u+v) & =(y,y,u)(xy^{-1},0,v) & =a_u(y) b_v(xy^{-1}) \\
(x,y,u-v) & =(y,y,u)(xy^{-1},0,-v) & =a_u(y) b^{-1}_v(x^{-1}y) \\
(x,y,-u+v) & =(-y^{-1},y,-u)(-xy,0,v) & =a^{-1}_u(-y) b_v(-xy) \\
(x,y,-u-v) & =(-y^{-1},y,-u)(-xy,0,-v) & =a^{-1}_u(-y) b^{-1}_v(-x^{-1}y^{-1}) 
\end{array}
\]
In case $y=0$, introducing also $b_u(1)$ and $b_u^{-1}(1)$ for $u\in U$,
\[
\begin{array} {l l l}
(x,0,u+v) & =(1,0,u)(x,0,v) & =b_u(1) b_v(x) \\
(x,0,u-v) & =(1,0,u)(x,0,-v) & =b_u(1) b^{-1}_v(x^{-1}) \\
(x,0,-u+v) & =(1,0,-u)(x,0,v) & =b^{-1}_u(1) b_v(x) \\
(x,0,-u-v) & =(1,0,-u)(x,0,-v) & =b^{-1}_u(1) b^{-1}_v(x^{-1}) 
\end{array}
\]
Finally consider $A_u$ and $A_{u'}$ for $u, u' \in U, u\ne u'$. For any $x\in F^*\setminus \{1\}$, $y\in F^*$ we have
\[
\begin{array} {l l}
(x,y,u-u') & =(xy/(x-1),xy/(x-1),u)((x-1)/y,-y/(x-1),-u')\\
 & =a_u(xy/(x-1)) a^{-1}_{u'}(y/(x-1)) \\
\end{array}
\]
In case $x=1$, where $w'$ is any fixed element of $W$, and introducing also $b_{u-u'-w'}(1)$,
\[
\begin{array} {l l l}
(1,y,u-u') & =(1,y,w')(1,0,u-u'-w') & =c_{w'}(y) b_{u-u'-w'}(1) \\
\end{array}
\]
And in case $y=0$, we consider $(x,0,u-u')=b_{u-u'}(x)$. If $u-u'=v$ for some $v\in V$ then this is immediately covered by $b_v(x)$. Otherwise if $u-u'=v+v'$ for some $v, v'\in V$ then
\[
\begin{array} {l l l l}
(x,0,u-u') & =(x,0,v+v') & =(x,0,v)(1,0,v') & =b_v(x) b_{v'}(1) \\
\end{array}
\]
and similarly for $u-u'=v-v'$ or $-v-v'$. For any remaining uncovered cases $b_{u-u'}(x)$ it would be necessary to introduce $b_u(x)$ for any $x\in F^*$, using the construction
\[
\begin{array} {l l l}
(x,0,u-u') & =(x,0,u)(1,0,-u') & =b_u(x) b_{u'}^{-1}(1). \\
\end{array}
\]
From the above, we see that each pair $A_u B_v$, $A_u C_w$, $B_v C_w$ can generate all elements of $G$ containing up to four values of $\Z_n$ with a limited number of exceptions that are covered by the additionally introduced elements as shown above. Similarly each pair $A_u A_{u'}$ generates all elements of $G$ containing up to two values of $\Z_n$. On the other hand the first coordinate of the product of an element of $C_w$ with an element of $C_{w'}$ is always 1, and the second coordinate of the product of an element of $B_v$ with an element of $B_{v'}$ is always 0. So these combinations do not contribute to an efficient covering of $G$. Summing the elements listed in the cases above gives the following lower bound for the degree $d$ of the corresponding circulant graph:
\[
\begin{array} {l l l}
d & =|X| & \ge 2g_a(p-1)+2g_b(p-1)+2g_cp+2g_a+g_a(g_a-1) \\
   &         & = 2m(p-1)+2g_c+g_a+g^2_a , \mbox{ as } m=g_a+g_b+g_c
\end{array}
\]
An upper bound for the value of $n$ is obtained by assuming there is no duplication between the values of $s+t$, $s-t$, $-s+t$ and $-s-t$ across all $s, t \in \{u, v, w\}$ for the given combinations of $A_u$, $B_v$ and $C_w$, with the exception that $(x,y,0)$ can be created from two elements of any $A_u$. Thus the upper bound for $n$ is given by
\[
\begin{array} {l l}
n & \le 4g_ag_b+4g_bg_c+4g_ag_c+g_a(g_a-1)+1 \\
 & =4g_am+4g_bm-3g^2_a-4g^2_b-4g_ag_b-g_a+1, \mbox{ as } m=g_a+g_b+g_c
\end{array}
\]

This is a maximum when the partial derivatives with respect to $g_a$ and $g_b$ are zero. We have
$\partial n/\partial g_a=4m-6g_a-4g_b-1=0$ and $\partial n/\partial g_b=4m-8g_b-4g_a=0$.
Thus $g_a=(2m-1)/4$ and $g_b=g_c=(2m+1)/8$, and hence $n\le (12m^2-4m+9)/8$.
Also we have $d\ge 2mp-5m/4+1/16+m^2/4$, so that $p\le d/2m+5/8-1/(32m)-m/8$. Thus we find
\[
\begin{array} {l l}
|G| & \le (p-1)p(12m^2-4m+9)/8 \\
 & =[3/8-(4m-9)/(32m^2)]d^2+O(d) \\
 & \le (3/8) d^2+O(d), \mbox{ for } m>2
\end{array}
\]
The exceptional cases, where $m \le 2 $, are easily evaluated separately, see Table \ref{table:1}.

If instead $X$ also includes the set $C_0$, so that $d=|X|\ge 2g_a(p-1)+2g_b(p-1)+(2g_c+1)p+2g_a+g_a(g_a-1)$, then we find that the upper bound for $n$ is given by
\[
\begin{array} {l}
n\le 4g_ag_b+4g_bg_c+4g_ag_c+2g_a+2g_b+g_a(g_a-1)+1
\end{array}
\]
In this case partial differentiation gives a maximum value at $g_a=m/2$, $g_b=(m+1)/4$ and $g_c=(m-1)/4$. Then $n\le (6m^2+4m+5)/4$. Also $d\ge (2m+1)p-m-1/2+m^2/4$, so that $p\le d/(2m+1)+1/2-m^2/(8m+2)$. In this case $|G| \le [(6m^2+4m+5)/(4(2m+1)^2)]d^2+O(d) \le (3/8)d^2+O(d)$ for $m>3$. Again, the exceptional cases, where $m \le 3 $, are evaluated separately.
\end{proof}

In both cases, with and without $C_0$ as a subset of the connection set $X$, the optimum values of $g_a$, $g_b$ and $g_c$ determined by differentiation are often not integral and the resultant value of $n$ is never integral. However in practice it appears possible to find values such that $n$ achieves the highest odd integer below the calculated value. As stated earlier, this upper bound assumes it is possible to find a set of values for the $u$, $v$ and $w$ such that none of the pairwise combinations are duplicates. This has been found to be possible in every case for $m\le 3$ and also for the case $m=5$ without $C_0$, but for no other values checked up to $m=8$. A summary of the best results is presented in Table \ref{table:1}. 

\begin{table} [!htb]
\small
\caption{\small Upper bounds and extremal values for the order of the circulant graph for $l\le 17$ } 
\centering 
\setlength {\tabcolsep} {8pt}
\begin{tabular}{ @ { } c c c c c c c c } 
\noalign {\vskip 1mm}  
\hline\hline 
\noalign {\vskip 1mm}  
Pairs & Sets &\multicolumn {2} {c} {Upper bound for $n$} & \multicolumn {2} {l} {Quadratic coefficient} & Extremal & Quadratic\\
$m$ & $l$ & Real & Integer & Fraction & Decimal & value, $n$ & coefficient\\ 
\hline
\noalign {\vskip 1mm}  

1 & 3 & 3.75 & 3 & 1/3 & 0.333 & 3 & 0.333 \\
2 & 4 & 6.125 & 5 & 5/16 & 0.313 & 5 & 0.313 \\
2 & 5 & 9.25 & 9 & 9/25 & 0.360 & 9 & 0.360 \\
3 & 6 & 13.125 & 13 & 13/36 & 0.361 & 13 & 0.361 \\
3 & 7 & 17.75 & 17 & 17/49 & 0.347 & 17 & 0.347 \\
4 & 8 & 23.125 & 23 & 23/64 & 0.359 & 21 & 0.328 \\
4 & 9 & 29.25 & 29 & 29/81 & 0.358 & 27 & 0.333 \\
5 & 10 & 36.125 & 35 & 35/100 & 0.350 & 35 & 0.350 \\
5 & 11 & 43.75 & 43 & 43/121 & 0.355 & 41 & 0.339 \\
6 & 12 & 52.125 & 51 & 51/144 & 0.354 & 49 & 0.340 \\
6 & 13 & 61.25 & 61 & 61/169 & 0.361 & 57 & 0.337 \\
7 & 14 & 71.125 & 71 & 71/196 & 0.362 & 65 & 0.332 \\
7 & 15 & 81.75 & 81 & 81/225 & 0.360 & 75 & 0.333 \\
8 & 16 & 93.125 & 93 & 93/256 & 0.363 & 87 & 0.340 \\
8 & 17 & 105.25 & 105 & 105/289 & 0.363 & 97 & 0.336 \\

\hline
\end{tabular}
\label{table:1} 
\end{table}

\begin{table} [!htb]
\small
\caption{\small A solution for each extremal value of the order of the circulant graph for $l\le 17$ } 
\centering 
\setlength {\tabcolsep} {8pt}
\begin{tabular}{ @ { } c c c l l l } 
\noalign {\vskip 1mm}  
\hline\hline 
\noalign {\vskip 1mm}  
Pairs & Sets & Extremal & Values of $u$ & Values of $v$ & Values of $w$ \\
$m$ & $l$ & value, $n$ & for $A_u$ & for $B_v$ & for $C_w$ \\ 
\hline
\noalign {\vskip 1mm}  

2 & 5 & 9 & 1 & 3 & 0 \\
3 & 6 & 13 & 1 & 3 & 4 \\
3 & 7 & 17 & 1, 8 & 3 & 0 \\
4 & 8 & 21 & 2, 9 & 3 & 1 \\
4 & 9 & 27 & 5, 11 & 12, 13 & 0 \\
5 & 10 & 35 & 13, 16 & 7, 8 & 9 \\
5 & 11 & 41 & 7, 17 & 2, 13 & 0, 1 \\
6 & 12 & 49 & 19, 22, 23 & 1, 14 & 16 \\
6 & 13 & 57 & 10, 24 & 3, 18 & 0, 1, 2 \\
7 & 14 & 65 & 5, 11, 20, 27 & 6, 7 & 30 \\
7 & 15 & 75 & 11, 20, 33 & 3, 26 & 0, 1, 2 \\
8 & 16 & 87 & 5, 20, 29 & 40, 41, 42 & 13, 39 \\
8 & 17 & 97 & 4, 31, 39 & 23, 24, 25 & 0, 13, 26 \\

\hline
\end{tabular}
\label{table:2} 
\end{table}

In most cases there are multiple solutions for the connection set, including with different values for $g_a$, $g_b$ and $g_c$. Table \ref{table:2} shows one solution for each case. For $l=5$ this is the solution given by Macbeth, \v{S}iagiov\'a and \v{S}ir\'{a}\v{n} \cite {Macbeth}. For $l=6$ this is the solution given by Vetr\'ik \cite {Vetrik}.
The quadratic coefficient of Vetr\'ik's solution, $13/36$, is likely to be the best possible with this form of construction. Although for a large enough value for $l$ it is possible for the upper bound to be arbitrarily close to $3/8$, we see from Table \ref{table:1} that the proportion of duplicates appears to increase with $l$, reaching about 7\% by $l=17$.


\section{Abelian Cayley graphs for groups of the form $H=F^* \times F^+ \times \Z_n$}

The circulant graph construction in the previous section can be extended to Abelian Cayley graphs $C(H,X)$ where $H=F^* \times F^+ \times \Z_n$ by relaxing the requirement that $n$ is coprime with $p$ and $p-1$, for $p$ a prime power, so that the connection set can include the self-inverse set $B_{n/2}$, thus requiring that $n$ be even. In Macbeth, \v{S}iagiov\'a and \v{S}ir\'{a}\v{n}'s construction with $n=6$ the connection set $X$ is comprised of the sets $A_1, B_3$ and $C_0^*$ along with two other elements, noting that $B_3$ is $B_{n/2}$, and hence $|X|=4p-2$. It is relatively straightforward to prove that any element of $H$ can be expressed as the sum of at most two elements of $X$ so that the resultant Cayley graph $C=(H,X)$ has diameter 2. As the degree of the graph $d=|X|$, we have $p=(d+2)/4$. Thus the order of the Cayley graph $|H|=6p(p-1)=(6/16)(d+2)(d-2)$, giving the quadratic coefficient $3/8$.

For the generalisation of this approach we again take the component sets $A_u, B_v, C_w^+$. For any $l\ge 4$ we consider different values of $n$ and for each $n$, the corresponding family of Abelian groups $H_l=F^* \times F^+ \times \Z_n$ for any prime power $p$. For each $H_l$ we consider generating sets $X$ comprised of all possible combinations of the sets $A_u, B_v, C_w^+$, always including $B_{n/2}$, along with a fixed number of other elements of $H_l$ such that $|X|=lp+\delta$ for some fixed $\delta$, with the condition that the Cayley graph $C(H_l,X)$ has diameter 2. As before we also consider the potential inclusion of the self-inverse set $C_0$ in the connection set. Thus as opposed to the circulant graph case, $C_0^+$ in included in $X$ if $l$ is odd and not included if $l$ is even.

As before, for any $l$ we define $n_l$ to be the largest value of $n$ for which such a Cayley graph $C(H,X)$ exists for all admissable $p$, and define $H_l=F^* \times F^+ \times \Z_{n_l}$. It follows that the order of $C(H_l,X)$ is $(n_l/l^2)d^2+O(d)$, and we denote this order by $AC_l(d,2)$. Finally we define the class of groups ${\cal H}=\{H_l: l\ge 3\}$ and $AC_{\cal H}(d,2)=\sup _{l\ge 3} {AC_l(d,2)}$. Theorem \ref{theorem:k2b} establishes an improved upper bound for such graphs, with quadratic coefficient 3/8.

\begin{theorem}
With $AC_{\cal H}(d,2)$ as defined above, $AC_{\cal H}(d,2) \le (3/8) d^2 + O(d)$.
\label{theorem:k2b}
\end{theorem}
\begin{proof}
Let $F^*$ be the multiplicative group and let $F^+$ be the additive group of the Galois field $GF(p)$, where $p$ is a prime power, and let $n$ be even. Let $H=F^* \times F^+ \times \Z_n$. Since $F^*$, $F^+$ and $\Z_n$ are cyclic groups, the group $H$ is Abelian.

Now consider the Cayley graph, $C(H,X)$, of the group $H$ with a connection set $X$ which includes the union of $A_u$, $B_v$ and $C_w$ for multiple values of $u$, $v$ and $w$, not equal to 0 or $n/2$, numbering $m$ in total. Consider $A_u$ for $u \in U$, $B_v$ for $v \in V$ and $C_w$ for $w \in W$ where $U, V$ and $W$ are index sets with $g_a=|U|$, $g_b=|V|$ and $g_c=|W|$, so that $g_a+g_b+g_c=m$.   $X$ is constructed to be inverse-closed so that $C(H,X)$ is an undirected Abelian Cayley graph of degree $d=|X|$. $X$ also includes $B_{n/2}$. We also consider including $C_0$. We define $l=2m+1$ if $C_0$ is not in $X$ and $l=2m+2$ if $C_0$ is in $X$. We now investigate how elements of $H$ may be constructed from pairs of the sets $A_u$, $B_v$, $C_w$.

First consider $B_v$ and $C_w$ for $v \in V$ and $w \in W$, where $v\ne n/2, w\ne 0$. For any $x\in F^*$, $y\in F^+$ we have
\[
\begin{array} {l l l}
(x,y,v+w) & =(x,0,v)(1,y,w) & =b_v(x) c_w(y) \\
(x,y,v-w) & =(x,0,v)(1,y,-w) & =b_v(x) c^{-1}_w(-y) \\
(x,y,-v+w) & =(x,0,-v)(1,y,w) & =b^{-1}_v(x^{-1}) c_w(y) \\
(x,y,-v-w) & =(x,0,-v)(1,y,-w) & =b^{-1}_v(x^{-1}) c^{-1}_w(-y) 
\end{array}
\]
In case $v=n/2, w\ne 0$
\[
\begin{array} {l l l}
(x,y,n/2+w) & =(x,0,n/2)(1,y,w) & =b_{n/2}(x) c_w(y) \\
(x,y,n/2-w) & =(x,0,n/2)(1,y,w) & =b_{n/2}(x) c_w(y)
\end{array}
\]
In case $w=0, v\ne n/2$
\[
\begin{array} {l l l}
(x,y,v) & =(x,0,v)(1,y,0) & =b_v(x) c_0(y) \\
(x,y,-v) & =(x,0,-v)(1,y,0) & =b^{-1}_v(x^{-1}) c_0(y)
\end{array}
\]
And in case $v=n/2, w=0$
\[
\begin{array} {l l l}
(x,y,n/2) & =(x,0,n/2)(1,y,0) & =b_{n/2}(x) c_0(y)
\end{array}
\]

Next consider $A_u$ and $C_w$ for $u \in U$ and $w \in W$. where $w\ne0$. For any $x\in F^*$, $y\in F^+$ we have
\[
\begin{array} {l l l}
(x,y,u+w) & =(x,x,u)(1,y-x,w) & =a_u(x) c_w(y-x) \\
(x,y,u-w) & =(x,x,u)(1,y-x,-w) & =a_u(x) c^{-1}_w(x-y) \\
(x,y,-u+w) & =(x,-x^{-1},-u)(1,y+x^{-1},w) & =a^{-1}_u(x^{-1}) c_w(y+x^{-1}) \\
(x,y,-u-w) & =(x,-x^{-1},-u)(1,y+x^{-1},-w) & =a^{-1}_u(x^{-1}) c^{-1}_w(-y-x^{-1}) 
\end{array}
\]
In case $w=0$
\[
\begin{array} {l l l}
(x,y,u) & =(x,x,u)(1,y-x,0) & =a_u(x) c_0(y-x) \\
(x,y,-u) & =(x,-x^{-1},-u)(1,y+x^{-1},0) & =a^{-1}_u(x^{-1}) c_0(y+x^{-1})
\end{array}
\]

Now consider $A_u$ and $B_v$ for $u \in U$ and $v \in V$ where $v\ne n/2$. For any $x\in F^*$, $y\in F^*$ we have
\[
\begin{array} {l l l}
(x,y,u+v) & =(y,y,u)(xy^{-1},0,v) & =a_u(y) b_v(xy^{-1}) \\
(x,y,u-v) & =(y,y,u)(xy^{-1},0,-v) & =a_u(y) b^{-1}_v(x^{-1}y) \\
(x,y,-u+v) & =(-y^{-1},y,-u)(-xy,0,v) & =a^{-1}_u(-y) b_v(-xy) \\
(x,y,-u-v) & =(-y^{-1},y,-u)(-xy,0,-v) & =a^{-1}_u(-y) b^{-1}_v(-x^{-1}y^{-1}) 
\end{array}
\]
In case $v=n/2$
\[
\begin{array} {l l l}
(x,y,n/2+u) & =(y,y,u)(xy^{-1},0,n/2) & =a_u(y) b_{n/2}(xy^{-1}) \\
(x,y,n/2-u) & =(-y^{-1},y,-u)(-xy,0,n/2) & =a^{-1}_u(y^{-1}) b_{n/2}(-xy)
\end{array}
\]

For $y=0$, where $v\ne n/2$, introducing also $b_u(1)$ and $b_u^{-1}(1)$ for $u\in U$,
\[
\begin{array} {l l l}
(x,0,u+v) & =(1,0,u)(x,0,v) & =b_u(1) b_v(x) \\
(x,0,u-v) & =(1,0,u)(x,0,-v) & =b_u(1) b^{-1}_v(x^{-1}) \\
(x,0,-u+v) & =(1,0,-u)(x,0,v) & =b^{-1}_u(1) b_v(x) \\
(x,0,-u-v) & =(1,0,-u)(x,0,-v) & =b^{-1}_u(1) b^{-1}_v(x^{-1}) 
\end{array}
\]
In case $v=n/2$
\[
\begin{array} {l l l}
(x,0,n/2+u) & =(1,0,u)(x,0,n/2) & =b_u(1) b_{n/2}(x) \\
(x,0,n/2-u) & =(1,0,-u)(x,0,n/2) & =b^{-1}_u(1) b_{n/2}(x)
\end{array}
\]

Finally consider $A_u$ and $A_u'$ for $u, u' \in U, u\ne u'$. For any $x\in F^*\setminus \{1\}$, $y\in F^*$ we have
\[
\begin{array} {l l}
(x,y,u-u') & =(xy/(x-1),xy/(x-1),u)((x-1)/y,-y/(x-1),-u')\\
 & =a_u(xy/(x-1)) a^{-1}_{u'}(y/(x-1)) \\
\end{array}
\]

In case $x=1$, where $w'$ is any fixed element of $W$, and introducing also $b_{u-u'-w'}(1)$,
\[
\begin{array} {l l l}
(1,y,u-u') & =(1,y,w')(1,0,u-u'-w') & =c_{w'}(y) b_{u-u'-w'}(1) \\
\end{array}
\]

And in case $y=0$, we consider $(x,0,u-u')=b_{u-u'}(x)$. If $u-u'=v$ for some $v\in V$ then this is immediately covered by $b_v(x)$. Otherwise if $u-u'=v+v'$ for some $v, v'\in V$ then
\[
\begin{array} {l l l l}
(x,0,u-u') & =(x,0,v+v') & =(x,0,v)(1,0,v') & =b_v(x) b_{v'}(1) \\
\end{array}
\]
and similarly for $u-u'=v-v'$ or $-v-v'$. For any remaining uncovered cases $b_{u-u'}(x)$ it would be necessary to introduce $b_u(x)$ for any $x\in F^*$, using the construction
\[
\begin{array} {l l l}
(x,0,u-u') & =(x,0,u)(1,0,-u') & =b_u(x) b_{u'}^{-1}(1). \\
\end{array}
\]
From the above, we see that each pair $A_u B_v$, $A_u C_w$, $B_v C_w$, for $v\ne n/2$ and $w\ne 0$, can generate all elements of $H$ with up to four values of $\Z_n$ with a limited number of exceptions that are covered as shown above. Similarly each pair $A_u A_{u'}$, $A_u C_w$, $B_v C_w$, for $u\ne u'$, $v=n/2$  or $w=0$, creates up to two. On the other hand the first coordinate of the product of an element of $C_w$ with an element of $C_{w'}$ is always 1, and the second coordinate of the product of an element of $B_v$ with an element of $B_{v'}$ is always 0. So these combinations do not contribute to an efficient covering of $H$.  Summing the elements listed in the cases above gives the following lower bound for the degree $d$ of the corresponding Abelian Cayley graph:
\[
\begin{array} {l l l}
d & =|X| & \ge 2g_a(p-1)+(2g_b+1)(p-1)+2g_cp+2g_a+g_a(g_a-1) \\
   &         & = (2m+1)(p-1)+2g_c+g_a+g^2_a , \mbox{ as } m=g_a+g_b+g_c
\end{array}
\]
An upper bound for the value of $n$ is obtained by assuming there is no duplication between the values of $s+t$, $s-t$, $-s+t$ and $-s-t$ across all $s, t \in \{u, v, w\}$ for the given combinations of $A_u$, $B_v$ and $C_w$, with the exception that $(x,y,0)$ can be created from two elements of any $A_u$. Thus the upper bound for $n$ is given by
\[
\begin{array} {l l}
n & =4g_ag_b+4g_bg_c+4g_ag_c+g_a(g_a-1)+2g_a+2g_c+1 \\
   & =4g_am+4g_bm-3g^2_a-4g^2_b-4g_ag_b+m-g_a-2g_b+1, \mbox{ as } m=g_a+g_b+g_c
\end{array}
\]
This is a maximum when the partial derivatives with respect to $g_a$ and $g_b$ are zero. We have
$\partial n/\partial g_a=4m-6g_a-4g_b-1=0$ and $\partial n/\partial g_b=4m-8g_b-4g_a-2=0$.
Thus $g_a=m/2$, $g_b=(m-1)/4$ and $g_c=(m+1)/4$, and hence $n\le (6m^2+4m+3)/4$.
Also we have $d\ge (2m+1)p+m^2/4-m-1/2$, so that $p\le d/(2m+1)-(m^2-4m-2)/(8m+4)$. Thus we find
\[
\begin{array} {l l}
|H| & \le (p-1)p(6m^2+4m+3)/4 \\
 & =[(6m^2+4m+3)/(4(2m+1)^2)]d^2+O(d) \\
 & \le (3/8) d^2+O(d), \mbox{ for } m\ge 1
\end{array}
\]
If instead $X$ also includes the set $C_0$, so that $d=|X|\ge 2g_a(p-1)+(2g_b+1)(p-1)+(2g_c+1)p+2g_a+g_a(g_a-1)$, then we find that the upper bound for $n$ is given by
\[
\begin{array} {l l}
n & \le 4g_ag_b+4g_bg_c+4g_ag_c+4g_a+2g_b+2g_c+g_a(g_a-1)+2 \\
   & =4g_am+4g_bm-3g^2_a-4g^2_b-4g_ag_b+2m+g_a+2
\end{array}
\]
In this case partial differentiation gives a maximum value at $g_a=(2m+1)/4$, $g_b=g_c=(2m-1)/8$. Then $n=(12m^2+20m+17)/8$. Also $d=2(m+1)p+m^2/4-3m/4-15/16$, so that $p=d/2(m+1)-(4m^2-12m-15)/32(m+1)$, and we have $|H|\le(3/8)d^2+O(d)$ for $m\ge 1$.
\end{proof}

As for the circulant graph case, the optimum values of $g_a$, $g_b$ and $g_c$ determined by differentiation are often not integral and the resultant value of $n$ is never integral. The upper bound is then the largest even number below the calculated value. This is only achievable if it is possible to find a set of values for the $u$, $v$ and $w$ such that none of the pairwise combinations are duplicates.
This has been found to be possible in every case for $m\le 3$, as for the circulant graphs, but for no higher values checked up to $m=7$. A summary of the best results is presented in Table \ref{table:3}.

\begin{table} [!htb]
\small
\caption{\small Upper bounds and extremal values for the order of the Abelian Cayley graph for $l\le 16$ } 
\centering 
\setlength {\tabcolsep} {8pt}
\begin{tabular}{ @ { } c c c c c c c c } 
\noalign {\vskip 1mm}  
\hline\hline 
\noalign {\vskip 1mm}  
Pairs & Sets &\multicolumn {2} {c} {Upper bound for $n$} & \multicolumn {2} {l} {Quadratic coefficient} & Extremal & Quadratic\\
$m$ & $l$ & Real & Integer & Fraction & Decimal & value, $n$ & coefficient\\ 
\hline
\noalign {\vskip 1mm}  

1 & 4 & 6.125 & 6 & 3/8 & 0.375 & 6 & 0.375 \\
2 & 5 & 8.75 & 8 & 8/25 & 0.320 & 8 & 0.320 \\
2 & 6 & 13.125 & 12 & 1/3 & 0.333 & 12 & 0.333 \\
3 & 7 & 17.25 & 16 & 16/49 & 0.327 & 16 & 0.327 \\
3 & 8 & 23.125 & 22 & 11/32 & 0.344 & 22 & 0.344 \\
4 & 9 & 28.75 & 28 & 28/81 & 0.346 & 26 & 0.321 \\
4 & 10 & 36.125 & 36 & 9/25 & 0.360 & 34 & 0.340 \\
5 & 11 & 43.25 & 42 & 42/121 & 0.347 & 40 & 0.331 \\
5 & 12 & 52.125 & 52 & 13/36 & 0.361 & 48 & 0.333 \\
6 & 13 & 60.75 & 60 & 60/169 & 0.355 & 56 & 0.331 \\
6 & 14 & 71.125 & 70 & 35/98 & 0.357 & 66 & 0.337 \\
7 & 15 & 81.25 & 80 & 16/45 & 0.356 & 72 & 0.320 \\
7 & 16 & 93.125 & 92 & 23/64 & 0.359 & 86 & 0.336 \\

\hline
\end{tabular}
\label{table:3} 
\end{table}

It is interesting to note that for $l=6$ we have extremal value $n=12$ with a quadratic coefficient of 1/3, whereas the corresponding circulant graphs have a higher extremal value, $n=13$ with coefficient 13/36. This is because the Abelian Cayley graph connection set is defined to include the two self-inverse sets $B_{n/2}$ and $C_0$ along with two pairs of non self-inverse sets, whereas the circulant graph connection set is comprised of three such pairs and neither of the self-inverse sets. Without the requirement to include the set $B_{n/2}$, the extremal Abelian Cayley graph with this construction would be the circulant graph.
In most cases there are multiple solutions for the connection set, including with different values for $g_a$, $g_b$ and $g_c$. Table \ref{table:4} shows one solution for each case. For $l=4$ this is the solution given by Macbeth, \v{S}iagiov\'a and \v{S}ir\'{a}\v{n} \cite {Macbeth}. The quadratic coefficient of this solution, $3/8$, is likely to be unmatched for any other value of $l$ with this form of construction. Although the upper bound is $3/8$, we see from Table \ref{table:3} that the proportion of duplicates appears to increase with $l$, reaching about 7\% by $l=16$. 

\begin{table} [!htb]
\small
\caption{\small A solution for each extremal value of the order of the Abelian Cayley graph for $l\le 16$ } 
\centering 
\setlength {\tabcolsep} {8pt}
\begin{tabular}{ @ { } c c c l l l } 
\noalign {\vskip 1mm}  
\hline\hline 
\noalign {\vskip 1mm}  
Pairs & Sets & Extremal & Values of $u$ & Values of $v$ & Values of $w$ \\
$m$ & $l$ & value, $n$ & for $A_u$ & for $B_v$ & for $C_w$ \\ 
\hline
\noalign {\vskip 1mm}  

1 & 4 & 6 & 1 & 3 & 0 \\
2 & 5 & 8 & 1 & 4 & 3 \\
2 & 6 & 12 & 1 & 6 & 0, 3 \\
3 & 7 & 16 & 1, 6 & 8 & 2 \\
3 & 8 & 22 & 2, 7 & 11 & 0, 1 \\
4 & 9 & 26 & 2, 9 & 13 & 1, 4 \\
4 & 10 & 34 & 1, 8 & 2, 17 & 0, 13 \\
5 & 11 & 40 & 12, 17, 18 & 1, 20 & 8 \\
5 & 12 & 48 & 10, 18 & 3, 24 & 0, 1, 2  \\
6 & 13 & 56 & 3, 4, 11 & 1, 26, 28 & 17 \\
6 & 14 & 66 & 12, 26 & 3, 18, 33 & 0, 1, 2 \\
7 & 15 & 72 & 3, 4, 11, 19 & 1, 34, 36 & 25  \\
7 & 16 & 86 & 11, 26, 36 & 3, 20, 43 & 0, 1, 2  \\

\hline
\end{tabular}
\label{table:4} 
\end{table}


\section{Abelian Cayley graphs for groups of the form $K=\Z_s \times \Z_t \times \Z_n$}

An alternative approach for the construction of an Abelian Cayley graph is to consider the direct product of three unrestricted cyclic groups $K=\Z_s \times \Z_t \times \Z_n$ with an appropriate connection set $X$. The approach of the previous section is adapted accordingly. First we need to adjust the definition of the component sets of the connection set.

\begin{definition}
For any $s, t, n \in \mathbb{N}, x\in \Z_s, y\in \Z_t, v, w\in \Z_n$, let $b_v(x)=(x,0,v)$ and $c_w(y)=(0,y,w)$, with $b^{-1}_v(x)=(-x,0,-v)$, and $c^{-1}_w(y)=(0,-y,-w)$.
For any $v, w\in \Z_n$ we also define $B_v=\{ b_v(x)$, $b^{-1}_v(x)$: $x\in \Z_s\}$ and $C_w=\{ c_w(y)$, $c^{-1}_w(y)$: $y\in \Z_t\}$.
\end{definition}

In this case, the sets $A_u$ are undefined as there is no fixed relation between $s$ and $t$. Hence only $B_v$ and $C_W$ may be considered for inclusion in the connection set $X$. For $v\ne 0$, $B_v$ has size $2s$ and for $w\ne 0$, $C_w$ has size $2t$. The size of $B_0$ is $s$, and of $C_0$ is $t$. Without loss of generality we consider the optional inclusion of $C_0$ and always omit $B_0$. As before, for any $l\ge 4$ we consider different values of $n$ and for each $n$, the corresponding family of Abelian groups $K_l=\Z_s \times \Z_t \times \Z_n$ for any values of $s, t$. For each $K_l$ we let $m=\lfloor l/2 \rfloor$ and consider generating sets $X$ comprised of all possible combinations of the sets $B_v, C_w$, with $v, w\ne 0$, along with a fixed number of other elements of $K_l$ such that the total number of sets $B_v, C_w$ is $m$, and including $C_0$ if $l$ is odd, with the condition that the Cayley graph $C(K_l,X)$ has diameter 2. For any $l$ we define $n_l$ to be the largest value of $n$ for which such a Cayley graph $C(K,X)$ exists, and define $K_l=\Z_s \times \Z_t \times \Z_{n_l}$. We denote the order of $C(K_l,X)$ by $UAC_l(d,2)$. Finally we define the class of groups ${\cal K}=\{K_l: l\ge 4\}$ and $UAC_{\cal K}(d,2)=\sup _{l\ge 4} {UAC_l(d,2)}$. Theorem \ref{theorem:k2c} establishes an improved upper bound for such graphs, with quadratic coefficient 1/4.

\begin{theorem}
With $UAC_{\cal K}(d,2)$ as defined above, $UAC_{\cal K}(d,2) \le (1/4) d^2 + O(d)$.

\label{theorem:k2c}
\end{theorem}
\begin{proof}

Let $\Z_s$, $\Z_t$ and $\Z_n$ be the additive cyclic groups of order $s$, $t$ and $n$ respectively. Let $K=\Z_s \times \Z_t \times \Z_n$. Then $K$ is an Abelian group. Now consider the Cayley graph, $C(K,X)$, of the group $K$ with a connection set $X$ which includes the union of $B_v$ and $C_w$ for multiple non-zero values of $v$ and $w$, numbering $m$ in total. Consider $B_v$ for $v \in V$ and $C_w$ for $w \in W$ where $V$ and $W$ are index sets with $g_b=|V|$ and $g_c=|W|$, so that $g_b+g_c=m$. $X$ is constructed to be self-inverse so that $C(K,X)$ is an undirected circulant graph of degree $d=|X|$. The inverse-closed set $C_0$ is also included in $X$, and we define $l=2m+1$. We now investigate how elements of $K$ may be constructed from pairs of the sets $B_v$, $C_w$.

As for the previous case of Abelian Cayley graphs, consider $B_v$ and $C_w$ for $v \in V$ and $w \in W$. For any $x\in \Z_s$, $y\in \Z_t$ we have
\[
\begin{array} {l l l}
(x,y,v+w) & =(x,0,v)(0,y,w) & =b_v(x) c_w(y) \\
(x,y,v-w) & =(x,0,v)(0,y,-w) & =b_v(x) c^{-1}_w(-y) \\
(x,y,-v+w) & =(x,0,-v)(0,y,w) & =b^{-1}_v(-x) c_w(y) \\
(x,y,-v-w) & =(x,0,-v)(0,y,-w) & =b^{-1}_v(-x) c^{-1}_w(-y) 
\end{array}
\]

Hence each pair $B_v C_w$ can generate all elements of $K$ with up to four values of $\Z_n$ for $w\ne 0$ and up to two values of $\Z_n$ for $w=0$. On the other hand the first coordinate of the sum of an element of $C_w$ and of $C_{w'}$ is always 0, as is the second coordinate of the sum of an element of $B_v$ and of $B_{v'}$. So these combinations do not contribute to an efficient covering of $K$.  Summing the elements listed in the cases above gives the following lower bound for the degree $d$ of the corresponding Abelian Cayley graph:
\[
\begin{array} {l l l}
d & =|X| & \ge  2g_bs+(2g_c+1)t
\end{array}
\]
An upper bound for the value of $n$ is obtained by assuming there is no duplication between the values of $v+w$, $v-w$, $-v+w$ and $-v-w$ across all $v, w$ for the given combinations of $B_v$ and $C_w$. Thus the upper bound for $n$ is given by
\[
\begin{array} {l}
n \le 4g_bg_c+2g_b =4g_bm-4g^2_b+2g_b, \mbox{ as } m=g_b+g_c 
\end{array}
\]

This is a maximum when the derivative with respect to $g_b$ is zero. We have $dn/dg_a=4m-8g_b+2=0$.
Thus $g_b=(2m+1)/4$ and $g_c=(2m-1)/4$, and hence $n\le  m^2+m+1/4$.
Also we have $d=|X|= (2m+1)s/2+(2m+1)t/2=(s+t)(2m+1)/2$, so that $m=d/(s+t)-1/2$.

Then as an upper bound, we have $|K|\le st(m^2+m+1/4)=(st/(s+t)^2)d^2$, which is maximal when $s=t$. Thus $|K|\le(1/4)d^2+O(d)$.
\end{proof}


\section{Extending the validity to any degree above a threshold}

We note that the graphs from both constructions by Macbeth, \v{S}iagiov\'a and \v{S}ir\'a\v{n} and from Vetr\'ik's are only established for values of the degree which are a linear function of a sequence of primes belonging to a prescribed congruence class or set of classes. For example, for Macbeth, \v{S}iagiov\'a and \v{S}ir\'a\v{n}'s circulant graph construction the degree $d=5p-3$ where the prime $p\equiv 2\pmod3$. A recent paper by Cullinan and Hajir \cite{Cullinan} defines a method of identifying a bound on the size of interval which will always contain at least one prime of a prescribed congruence class. This builds on an earlier paper by Ramar\'e and Rumely \cite{Ramare} which includes a table, Table 1, that defines triples $(k, x_0, \epsilon)$ where $k$ is the modulo of the congruence class, $x_0$ is a threshold minimum and $\epsilon$ is a corresponding factor. For any such triple Cullinan and Hajir established that for any $x>x_0$ and $\delta>2\epsilon/(1-\epsilon)$ the interval $(x, x(1+\delta)]$ will contain at least one prime $p\equiv a\pmod k$ for any $a$ coprime with $k$. Therefore in each of the graph constructions discussed, for any sufficiently large degree $d$ it is possible to find a prime of the correct congruence class such that the corresponding degree $d'<d(1+\delta)$. Within the table, $x_0$ takes four values, from $10^{10}$ to $10^{100}$, with the lowest value of $\epsilon$ corresponding to $x_0=10^{100}$. In Theorem \ref{theorem:ex} below we take the largest value for $x_0$ in order to define the largest possible lower bound on the asymptotic value of the quadratic coefficient for each construction which is valid for any degree above the corresponding threshold.

\begin{theorem}
For the Macbeth, \v{S}iagiov\'a and \v{S}ir\'a\v{n} circulant graph construction with quadratic coefficient $9/25=0.3600$, the order $n$ of the graph of degree $d$ is bounded below by $n>0.3581d^2+O(d)$ for any $d>5\times 10^{100}$. For the Vetr\'ik circulant graph construction with quadratic coefficient $13/36\approx 0.3611$, the order $n$ of the graph of degree $d$ satisfies $n>0.3582d^2+O(d)$ for any $d>6\times 10^{100}$. For the Macbeth, \v{S}iagiov\'a and \v{S}ir\'a\v{n} Abelian Cayley graph construction with quadratic coefficient $3/8=0.3750$, the order $n$ of the graph of degree $d$ satisfies $n>0.3749d^2+O(d)$ for any $d>4\times 10^{100}$.

\label{theorem:ex}
\end{theorem}
\begin{proof}
For the Macbeth, \v{S}iagiov\'a and \v{S}ir\'a\v{n} circulant graph construction with quadratic coefficient $9/25$ we take the triple $k=3, x_0=10^{100}, \epsilon=0.001310$ from Table 1 of \cite{Ramare}, giving a value of $\delta=0.002623$. Then $9/25 \times 1/(1+\delta)^2\approx 0.35811$. This is valid for $p>10^{100}$, and hence for $d>5\times 10^{100}$ as $d<d'=5p-3$.
For the Vetr\'ik circulant graph construction with quadratic coefficient $13/36$ we take the triple $k=13, x_0=10^{100}, \epsilon=0.002020$ from the table, giving a value of $\delta=0.004048$. Then $13/36 \times 1/(1+\delta)^2\approx 0.35821$, and we note that $d'=6p-2$.
For the Macbeth, \v{S}iagiov\'a and \v{S}ir\'a\v{n} Abelian Cayley graph construction with quadratic coefficient $3/8$ we take the triple $k=1, x_0=10^{100}, \epsilon=0.000001$ from the table, giving a value of $\delta=0.000002$. Then $3/8 \times 1/(1+\delta)^2\approx 0.37499$, and we note that $d'=4p-2$.
\end{proof}


\section{Conclusion}

For circulant graphs of diameter 2, employing the method of construction with the direct product of the additive and multiplicative groups of a Galois field and cyclic group of any order, it proves to be impossible to achieve a quadratic coefficient better than $3/8$, thus improving the upper bound for this construction method from $1/2$ to $3/8$. For Abelian Cayley graphs constructed in the same way the upper bound on the quadratic coefficient is also improved to a value of 3/8. Applying Cullinan and Hajir's property of intervals containing a prime of a prescribed congruence class, the asymptotic value of the quadratic coefficient valid for every degree above a threshold exceeds 0.358 for the two circulant graph construction and remains at 0.375 to three significant figures for the Abelian Cayley graph construction. Also the quadratic coefficient upper bound for the direct product of three unconstrained cyclic groups with the given construction method is improved to equal the general lower bound of $1/4$.

\begin{table} [!hb]
\small
\caption{\small Order and generator set for extremal circulant graphs of diameter 2 up to degree 23} 
\centering 
\setlength {\tabcolsep} {5pt}
\begin{tabular}{ @ { } c l l } 
\noalign {\vskip 1mm}  
\hline\hline 
\noalign {\vskip 1mm}  
Degree $d$ & Order $n$ & Generator set* \\
\hline
\noalign {\vskip 1mm}
2 & 5 &1  \\
3 & 8 &1  \\
4 & 13 &1, 5 \\
5 & 16 &1, 3  \\
6 & 21 &1, 2, 8 \\
7 & 26 &1, 2, 8  \\
8 & 35 &1, 6, 7, 10  \\
9 & 42 &1, 5, 14, 17 \\
10 & 51 &1, 2, 10, 16, 23  \\
11 & 56 &1, 2, 10, 15, 22  \\
12 & 67 &1, 2, 3, 13, 21, 30  \\
13 & 80 &1, 3, 9, 20, 25, 33  \\
14 & 90 &1, 4, 10, 17, 26, 29, 41  \\
15 & 96 &1, 2, 3, 14, 21, 31, 39  \\
16 & 112 &1, 4, 10, 17, 29, 36, 45, 52   \\
17 & 130 &1, 7, 26, 37, 47, 49, 52, 61   \\
18 & 138 &1, 9, 12, 15, 22, 42, 27, 51, 68  \\
19 & 156 &1, 15, 21, 23, 26, 33, 52, 61, 65   \\
20 & 171 &1, 11, 31, 36, 37, 50, 54, 47, 65, 81   \\
21 & 192 &1, 3, 15, 23, 32, 51, 57, 64, 85, 91 \\
22 & 210 &2, 7, 12, 18, 32, 35, 63, 70, 78, 91, 92  \\
23 & 216 & 1, 3, 5, 17, 27, 36, 43, 57, 72, 83, 95  \\
\hline
\noalign {\vskip 1mm}
 \multicolumn {3} {l} {* In many cases multiple isomorphism classes exist but only one for each order is shown here.} \\
 \multicolumn {3} {l} {For even degree the connection set is derived from the generator set by including the} \\
 \multicolumn {3} {l} {additive inverses.  For odd degree the connection set also includes the self-inverse element $n/2$.} \\
\end{tabular}
\label{table:k2} 
\end{table}

An interesting alternative perspective on the problem is provided by taking the known extremal solutions and determining the quadratic polynomial in the degree that gives a minimum least squares fit to their orders. For circulant graphs of diameter 2, extremal solutions up to degree $d=23$ have been confirmed by exhaustive computer search up to their respective upper bounds $M_{AC}(d,2)$ by Grahame Erskine and the author, see Table \ref{table:k2}. The resulting best fit equation was found to be $n=0.375d^2+0.961d+2.07$, which is a remarkably focused, albeit tentative, indication of support for a conjecture that the asymptotic value of the quadratic coefficient for the order of extremal solutions for large $d$ is in fact $3/8$. This leads to the challenge for further research: to develop an improved upper bound with quadratic coefficient less than $1/2$ and as close as possible to $3/8$ which is valid for general circulant graphs or Abelian Cayley graphs.


\end{document}